\documentclass[12pt,a4paper]{amsart}
\usepackage{amssymb,amscd}
\usepackage{mathrsfs}
\usepackage[mathcal]{eucal}
\usepackage{float}
\usepackage[stable]{footmisc}

\usepackage{xy}
\input xy
\xyoption{all}
\usepackage{pdflscape}
\usepackage{hyperref}

\usepackage[english]{babel}

\def\A{\mathbb{A}}
\def\B{\mathbb{B}}
\def\Q{\mathbb{Q}}
\def\R{\mathbb{R}}
\def\C{\mathbb{C}}
\def\F{\mathbb{F}}

\def\P{\mathbb{P}}

\let\myacute=\'

\def\<{\langle}
\def\>{\rangle}
\def\N{\mathbb{N}}
\def\Z{\mathbb{Z}}

\def\cL{\mathcal{L}}

\def \begindm {\begin{displaymath}}

\def \enddm {\end{displaymath}}

\def\C{\mathbb{C}}
\def\R{\mathbb{R}}
\def\Q{\mathbb{Q}}
\def\F{\mathbb{F}}

\def\cL{\mathcal{L}}

\def\cM{\mathcal{M}}
\def\cO{\mathcal{O}}

\newtheorem{thm}{Theorem}[section]

\newtheorem{cor}{Corollary}[section]
\newtheorem{prop}{Proposition}[section]
\newtheorem{claim}{Claim}[section]

\newtheorem{remark}{Remark}[section]

\newtheorem{fact}{Fact}[section]
\numberwithin{equation}{section}

\long\def\symbolfootnote[#1]#2{\begingroup\def\thefootnote{\fnsymbol{footnote}}\footnote[#1]{#2}\endgroup}

\title{Model Theory of Adeles {I}}
\author[J. Derakhshan]{Jamshid Derakhshan}
\address{St Hilda's College, University of Oxford, Cowley Place, Oxford OX4 1DY, and Mathematical Institute, Oxford, OX2 6GG, UK}
\email{derakhsh@maths.ox.ac.uk}

\author[A. Macintyre]{Angus Macintyre${}^{\dag}$}
\address{Queen Mary, University of London,
School of Mathematical Sciences, Queen Mary, University of London, Mile End Road, London E1 4NS, UK}
\email{angus@eecs.qmul.ac.uk}

\thanks{${}^{\dag}$Supported by a Leverhulme Emeritus Fellowship}

\begin{document}


\subjclass[2000]{}

\begin{abstract} 
We study the model theory of the ring of adeles of a number field. We obtain quantifier elimination results in 
the language of rings and some enrichments. We given consequences for definable subsets of the adeles, and their measures. 
\end{abstract}

\maketitle

\section{Introduction}\label{sec-introduction}

In this paper we study the model theory of the ring $\Bbb A_K$ 
of adeles of a number field $K$ in the language of rings and some enrichments. Our basic analysis is closely related to that 
of Weispfenning \cite{weisp-hab}, but our formalism is different. 

We consider $\Bbb A_K$ as a substructure of a 
generalized product in the sense of the classic 
paper \cite{FV} of Feferman-Vaught, but stressing that the associated Boolean 
system is living on the algebra of idempotents of $\Bbb A_K$. This enables us to internalize the generalized product structure 
on $\A_K$ within the language of rings or arithmetically significant enrichments of it, and obtain quantifier-elimination in natural languages.
We then show that the definable subsets of $\A_K^m$ are measurable, for any $m\geq 1$, and make a good start on the study of the numbers 
that can occur as normalized measures of definable subsets of the adeles.

\section{Basic notions}\label{basic}

\subsection{\it Valuations and absolute values} 

\

\

It is convenient for us to follow the well-known book by Cassels and Frohlich \cite{CF}, and in particular 
Cassels' chapter ``Global fields'' 
\cite{Cassels}, but will 
use only the number field case. 
We use ``K'' for a general number field. Note that Cassels' ``valuations'' $|.|$ are usually called absolute values. 

We are concerned only with three kinds of valuation on $K$ (up to equivalence of defining the same topology), namely:

(1) discrete non-archimedean, residue field finite of cardinality $q$,

(2) completion of $|.|$ is $\R$,

(3) completion of $|.|$ is $\Bbb C$.

In case 1, $|.|$ is normalized if $|\pi|=1/q$ where $\pi$ is a uniformizing element (i.e.\ $v(\pi)$ is minimal positive). 

In case 2, $|.|$ is the usual absolute value, and in case 3 $|.|$ is the square of the usual absolute value.

For us it will be convenient to write $v$ for the corresponding logarithmic valuation, defined by $|x|=p^{-v(x)}$, 
and $K_v$ for the completion. We write 
$\mathrm{Arch}(K)$ for the set of archimedean normalized valuations (a finite set), and 
$V_K$ for the set of all normalized valuations on $K$. We put $V_K^f=V_K \setminus \mathrm{Arch(K)}$, the set of 
non-archimedean valuations (or finite places).

We shall denote by $\Gamma$ the value group of a valued field and by $\cO_K:=\{x: |x|_v\leq 1\}$ and $\cM_K:=\{x: |x|_v\le 1\}$ the 
valuation ring of $K$ and its maximal 
ideal respectively. For a completion $K_v$ of $K$, where $v\notin \mathrm{Arch}(K)$, we shall denote $\cO_{K_v}$ by $\mathcal{O}_v$, and 
the residue field by $k_v$.

\subsection{\it Restricted direct products and measures}

\
 
\

Let $\Lambda$ denote an index set, and $\Lambda_{\infty}$ a fixed finite 
subset of $\Lambda$. Suppose that we are given, for each $\lambda\in \Lambda$, a locally compact space  
$G_{\lambda}$ and for all $\lambda\notin \Lambda$, 
a fixed compact and open subset $H_{\lambda}$ of $G_{\lambda}$. Then the 
restricted direct product of $G_{\lambda}$ with respect to $H_{\lambda}$ is defined to be the set of all elements 
$(x(\lambda))_{\lambda}\in \prod_{\lambda \in \Lambda} G_{\lambda}$ such that $x(\lambda)\in H_{\lambda}$ for almost 
all $\lambda$, and denoted $\prod'_{\lambda\in \Lambda} G_{\lambda}$.

Note that our notation $^{'}$ does not mention the $H_{\lambda}$ but of course the definition depends on $H_{\lambda}$. 
This restricted product carries the topology with a basis of open sets the sets 
$\prod_{\lambda} \Gamma_{\lambda}$, where $\Gamma_{\lambda}\subseteq G_{\lambda}$ is open for all $\lambda$, and 
$\Gamma_{\lambda}=H_{\lambda}$ for almost all $\lambda$. For a finite subset $S\subset \Lambda$ containing $\Lambda_{\infty}$, 
$G_S:=\prod_{\lambda\in S} G_{\lambda}\times \prod_{\lambda\notin S} H_{\lambda}$ is locally compact and open in $G$, and 
$G$ is the union of the $G_S$ over all finite subsets $S$ of $\Lambda$ containing $\Lambda_{\infty}$.

Let $\mu_{\lambda}$ denote a Haar measure on $G_{\lambda}$ such that $\mu_{\lambda}(H_{\lambda})=1$ 
for almost all $\lambda\notin \Lambda_{\infty}$. We define the product measure $\mu$ on 
$G$ to be the measure for which a basis of measurable sets consists of the sets 
$\prod_{\lambda} M_{\lambda}$, where $M_{\lambda}\subset G_{\lambda}$ is $\mu_{\lambda}$-measurable and 
$M_{\lambda}=H_{\lambda}$ for almost all $\lambda$, and $\mu(\prod_{\lambda} M_{\lambda})=\prod_{\lambda} 
\mu_{\lambda}(M_{\lambda})$. \

\subsection{\it The ring of adeles $\Bbb A_K$.} 

\

\

The adele ring $\Bbb A_K$ over a number field $K$ 
is the topological ring whose underlying topological space is the restricted direct product 
of the additive groups $K_v$ ($v$ normalized) with respect to the subgroups $\mathcal{O}_v$ with 
addition and multiplication defined componentwise. 
The restricted direct product 
$$\A_{K}^{fin}=\prod'_{v\notin \mathrm{Arch}(K)} K_v$$ 
is called the ring of finite adeles. One has 
$$\Bbb A_K=\prod_{v\in \mathrm{Arch}}K_v \times \Bbb A_{K}^{fin}.$$ 
A typical adele $a$ will be written as $(a(v))_{v}$. 

There is an embedding of $K$ into $\Bbb A_K$ sending $a\in K$ to the 
constant sequence $(a,a,\cdots)$. The image is called the ring of principal 
adeles, which we can identify with $K$. It is a discrete subspace of $\Bbb A_K$ with compact quotient $\A_K/K$. Note that if 
$K\subseteq L$ are global fields, then 
$$\A_L=\A_K \otimes_{K} L.$$

Each $K_v$ is a locally compact field, and carries a Haar measure $\mu_v$. 
These yield measures on $\A_K$ and $\A_K^{fin}$. 

The idele group $\Bbb I_K$ is the multiplicative group of units of $\Bbb A_K$.
It coincides with the restricted direct product of the multiplicative groups $K_v^*$ with respect to 
the groups of units $U_v$ of $\mathcal{O}_v$. 
 
It is most natural to give it the restricted direct product topology, 
which does not coincide with the subspace topology (cf.~\cite{CF}). The multiplicative group 
$K_v^*$ carries a Haar 
measure defined by the integral $\int f(x) |x|^{-1}dx$, for every Borel function $f(x):K_v\rightarrow \Bbb C$, 
where $dx$ is an additive Haar measure on $K_v$. This yields a measure on the ideles 
$\Bbb I_K$.

Certain normalizations of the above measures for $\A_K$ and $\Bbb I_K$ have been of importance in number theory, in 
Tate's thesis in \cite{CF}. We shall show that Tate's normalization factors are volumes of sets that are definable $K_v$ uniformly in $v$, in some enrichment of the ring language (cf. Section \ref{sec-def-sets}.

\subsection{\it Idempotents and support}\label{sec-id} 

\

\

Let $\B_K$ denote the set of idempotents in $\A_K$, namely, 
all elements $a\in \A_K$ such that $a^2=a$. $\B_K$ is a Boolean algebra with the Boolean 
operations defined by 
$$e\wedge f=ef,$$ 
$$e\vee f=1-(1-e)(1-f)=e+f-ef,$$ 
and 
$$\bar e=1-e.$$ 
Clearly, $\B_K$ is a definable subset of $\Bbb A_K$ in the language of rings. 

Note that there is a correspondence between subsets of $V_K$ and 
idempotents $e$ in $\Bbb A_K$ given by 

$$X \longrightarrow e_X,$$ where 
$e_X(v)=1$ if $v\in X$, and $e_X=0$ if $v\notin X$. Clearly $e_X\in \Bbb A_K$. Conversely, 
if $e\in \Bbb A_K$ is idempotent, let $X=\{v: e(v)=1\}$. Then $e=e_{X}$. 

Note that minimal idempotents $e$ correspond to normalized valuations $v_e$, and vice-versa, 
$v$ corresponds to $e_{\{v\}}$ above. Note that 

$$\Bbb A_K/(1-e)\Bbb A_K \cong e\Bbb A_K\cong K_{v_e}.$$

Let $\cL_{\rm{rings}}$ denote the language of rings $\{+,.,0,1)\}$. For any $\cL_{\rm{rings}}$-formula 
$\Phi(x_1,\dots,x_n)$ define $Loc(\Phi)$ as the set of all 
$$(e,a_1,\dots,a_n)\in\Bbb A_K^{n+1}$$
such that $e$ is a minimal idempotent and 
$$e\Bbb A_K \models \Phi(ea_1,\dots,ea_n).$$

Note that $e\Bbb A_K$ is 
a subring of $\Bbb A_K$ with $e$ as its unit, is definable with the parameter $e$, and 
$$e\A_K\models \Phi(ea_1,\dots,ea_n)$$
if and only if
$$\A_K\models \Phi(ea_1,\dots,ea_n).$$
Clearly $Loc(\Phi)$ is a definable subset of $\Bbb A_K^{n+1}$. 

For $a_1,\dots,a_n\in \Bbb A(K)$, define 
$$[[\Phi(a_1,\dots,a_n)]]$$ as the supremum of all the minimal idempotents $e$ in $\Bbb B_K$ such that 
$$(e,a_1,\dots,a_n)\in Loc(\Phi).$$
If the set of such minimal idempotents is empty, then 
$$[[\Phi(a_1,\dots,a_n)]]=0$$
We think of this as a Boolean value of $\Phi(a_1,\dots,a_n)$. For fixed $\Phi(x_1,\dots,x_n)$, the function 
$\Bbb A_K^n\rightarrow \Bbb A_K$ given by 
$$(a_1,\dots,a_n)\rightarrow [[\Phi(a_1,\dots,a_n)]]$$ 
is $\cL_{\rm{rings}}$-definable uniformly in $K$. 
The support of $a\in \Bbb A_K$, denoted $\mathrm{supp}(a)$, is defined as 
the Boolean value $[[\Phi(a)]]$ where $\Phi(x)$ is $x\neq 0$.

We let $\B_{K}^f$ denote the Boolean algebra of idempotents in $\A_{K}^{fin}$. Given 
$a_1,\cdots,a_n\in \A_{K}^{fin}$ and $\cL_{\rm{rings}}$-formula $\Phi(x_1,\cdots,x_n)$, the Boolean 
value 
$$[[\Phi(x_1,\cdots,x_n)]]^{fin}$$ is the supremum of all the 
minimal idempotents $e$ such that 
$$e\A_{K}^{fin} \models \Phi(ea_1,\dots,ea_n).$$ 
Note that this is $\cL_{\rm{rings}}$-definable in $\A_K$, and the map 
$$(\A_{K}^{fin})^n \rightarrow \A_{K}^{fin}$$
given by 
$$(x_1,\cdots,x_n) \rightarrow [[\Phi(x_1,\dots,x_n)]]$$ 
is $\cL_{\rm{rings}}$-definable uniformly 
in $K$. 

Let $\Psi^{\mathrm{Arch}}$ be a sentence that holds in $\R$ and $\C$ but does not hold in a 
non-archimedean local field (for example $\forall x \exists y (x=y^2 \vee -x=y^2)$). We call 
a minimal idempotent $e$ archimedean if $e\A_K\models \Psi^{\mathrm{Arch}}$, and 
non-archimedean otherwise. Let $e_{\infty}$ denote the supremum of all the archimedean minimal idempotents. 
$e_{\infty}$ is supported precisely on the set $\mathrm{Arch}(K)$, and 
$1-e_{\infty}$ precisely on the set of non-archimedean valuations. 

Let $\Psi_{\R}$ be a sentence that holds precisely in the $K_v$ which are isomorphic to $\R$ 
(e.g. $\Psi^{\mathrm{Arch}}\wedge \neg \exists y (y^2=-1)$), and 
$\Psi_{\C}$ a sentence that holds precisely in the $K_v$ which are isomorphic to $\C$ 
(e.g. $\Psi^{\mathrm{Arch}}\wedge \exists y (y^2=-1)$). We call a minimal idempotent real if 
$e\A_K\models \Psi_{\R}$, and complex if $e\A_K\models \Psi_{\C}$ . Let 
$e_{\R} $(resp.\ $e_{\C}$) denote the supremum of all the real (resp.~complex) minimal idempotents. 
Then $e_{\R}$ (resp. $e_{\C})$ are supported precisely on the set of 
$v$ such that $K_v$ is real  (resp. complex).

For $a_1,\cdots,a_n\in \Bbb A_k$ and $\cL_{\rm{rings}}$-formula $\Phi(x_1,\dots,x_n)$, we denote by  
$$[[\Phi(a_1,\cdots,a_n)]]^{real}$$
the supremum of all the minimal idempotents $e$ 
such that 
$$e\Bbb A_K\models \Psi_{\R} \wedge \Phi(ea_1,\cdots,ea_n),$$
and by
$$[[\Phi(a_1,\cdots,a_n)]]^{complex}$$
the supremum of all the minimal idempotents $e$ 
such that 
$$e\Bbb A_K\models \Psi_{\C} \wedge \Phi(ea_1,\cdots,ea_n).$$
We denote by 
$$[[\Phi(a_1,\cdots,a_n)]]^{na}$$
the supremum of all the minimal idempotents $e$ such that 
$$e\Bbb A_K\models \neg \Psi^{\mathrm{Arch}} \wedge \Phi(ea_1,\cdots,ea_n).$$ 
The functions $\A_K^n \rightarrow \A_K$ given by 
$$(a_1,\dots,a_n)\rightarrow [[\Phi(a_1,\cdots,a_n)]]^{real},$$
$$(a_1,\dots,a_n)\rightarrow [[\Phi(a_1,\cdots,a_n)]]^{complex},$$ and 
$$(a_1,\dots,a_n)\rightarrow [[\Phi(a_1,\cdots,a_n)]]^{na},$$ 
are all $\cL_{\rm{rings}}$-definable uniformly in $K$.

Note that one can not in general seperate the $K_v$ using a sentence of $\cL_{\rm{rings}}$ by their residue characteristic 
when the number field is not 
$\Q$ (the case of $\Q_p$ can be done using the uniform definition of the valuation rings $\Z_p$ from the language of rings by the results in \cite{CDLM} - note that this uniform definition holds more generally for all finite extensions of all $\Q_p$),
but one can do this for the class of $K_v$ which have 
residue characteristic $p$ using any $\cL_{\rm{rings}}$-definition of the valuation rings of $K_v$ (for example the one in \cite{CDLM}).

\section{Finite support idempotents}\label{fin-supp}

Among the idempotents $e$ there are the especially important ones:

i) $e$ with $\mathrm{supp}(e)$ of size $0,1,2\dots$

ii) $e$ of finite support.

We denote by $Fin_K$ the set of idempotents in $\A_K$ with finite support. It is an ideal in the Boolean algebra $\Bbb B_K$. When working with the finite adeles, $\A_K^{fin}$, we will use the same 
notation for the ideal of finite support idempotents in $\A_K^{fin}$. 

It had been known for some time that there is a uniform  $\cL_{\rm{rings}}$-definition of the valuation ring for all 
completions of number fields. See, for example, \cite{koenigsmann}.
However, only recently, and motivated by the objectives of the present paper, has precise information been obtained 
on the complexity of such a definition, in \cite{CDLM}, where the following theorem is proved.

\begin{thm}\cite[Theorem 2]{CDLM}\label{CDLM-th} There is an existential-universal formula in the language of rings that uniformly defines the valuation ring of all Henselian valued fields with finite or pseudo-finite residue field.
\end{thm}

Here we will be interested in applying this to the valuation rings of all the completions of a given number field (or of all number fields). However, the result is far more general 
as it applies to all Henselian valued fields which have higher rank valuations and arbitrary value groups and their ultraproducts.  

Let us denote by $\Phi_{val}(x)$ the formula from Theorem \ref{CDLM-th}.

\begin{thm}\label{fin-def} The ideal $\F_K$ is definable in 
$\Bbb A_K$ by an 
$\exists\forall\exists$-formula of the language of rings uniformly in $K$. The same is true for the ideal of finite support idempotents in $\A_K^{fin}$.
\end{thm}
\begin{proof} 
$\Phi_{val}(x)$ uniformly defines the valuation ring of $K_v$ 
for all number fields $K$ and all $v\in V_K^f$. 
We claim that $e\in \F_K$ if and only if 
$$\Bbb A_K\models\exists x\exists e_1\exists e_2 (e_1=e_1^2\wedge e_2=e_2^2\wedge e=e_1+e_2\wedge e_2=[[\neg\Phi_v(x)]]^{na}).$$
Indeed, let $a\in \Bbb A_K$ and idempotents $e_1,e_2$ satisfy $e=e_1+e_2$ and 
$e_2=[[\neg\Phi_{val}(a)]]^{na}$. Since $a$ is an adele, there are only finitely many $v\notin \mathrm{Arch}(K)$ satisfying 
$v(a(v))<0$. Thus $e_2$ has finite support. Since there are only finitely many archimedean valuations, 
$e$ has finite support.

Conversely, let $e$ be an idempotent with 
finite support. Let $\pi_v$ denote a uniformizing element of $K_v$. Define an adele $a$ as follows. We set $a(v)=1$ if $v\in \mathrm{Arch}(K)$. If $v\notin \mathrm{Arch}(K)$, we put 
$a(v)=\pi_{v}^{-1}$ if $e(v)=1$, and $a(v)=0$ if $e(v)=0$.  Then $a$ is clearly an adele.  
Let $e_2=[[\neg\Phi_{val}(a)]]^{na}$.  Then $e_2$ has finite support. Let $e_1=1-e_2$. 
Then $e_1$ has finite support and $e=e_1+e_2$.

This definition of $\Bbb F_K$ is uniform for all number fields $K$, since the definition 
$[[\Phi^{\mathrm{Val}}(x)]]^{na}$ and the definition $\Phi_{val}(x)$ 
are uniform for all number fields $K$. 

For the case of the finite adeles $\A_{K}^{fin}$ note that an idempotent $e$ has finite supprt if and only if
$$\A_{K}^{fin}\models \exists x~(e=[[\neg \Phi_{val}(x)]]).$$ 
Indeed, the right to left implication is 
clear. For the other implication, given a finite support $e$, define $a\in \A_{K}^{fin}$ by 
$a(v)=\pi_v^{-1}$ if $e(v)=1$, and 
$a(v)=0$ if $e(v)=0$. Then $e=[[\neg \Phi_{val}(a)]]$.\end{proof}

\section{Von Neumann regularity}

We call an element $a$ in a commutative ring von Neumann regular if the ideal generated by 
$a$ (which we denote by $(a)$) is generated by an idempotent. We call a ring von Neumann regular if every element is 
von Neumann regular. Given $a\in \Bbb A_K$, denote by $[[a]]_+$ the idempotent which is supported 
on the set 
$$\{v\in V_K: v(a(v))>0 \wedge a(v)\neq 0\},$$
and by $[[a]]_{+}^{fin}$ the idempotent which is supported 
on the set 
$$\{v\in V_K^{fin}: v(a(v))>0 \wedge a(v)\neq 0\}.$$

We shall denote by $supp(a)$ the set $\{v\in V_K: a(v)\neq 0\}$

\begin{prop}\label{v-n-reg} An element $a\in \Bbb A_K$ is von Neumann regular if and only if the idempotent 
$[[a]]_+$ has finite support.
\end{prop}
\begin{proof} Suppose that $[[a]]_+$ has finite support. Let $e=[[x\neq 0]](a)$. 
Define $c\in \A_K$ by $c(v)=a(v)^{-1}$ if $e(v)=1$, and $c(v)=0$ otherwise. It is clear that 
$a=ea$ and $e=ca$. Thus $a$ is von Neumann regular. 

Conversely, 
assume that $a=be$ and $e=ca$ for $b,c\in \A_K$ and idempotent $e$. 
Then $supp(a)=supp(e)$, hence $supp(a)=e$. Moreover $a(v)$ 
is invertible for all $v$ satisfying $e(v)=1$. Since $c$ is an adele, we deduce that 
$[[a]]_+$ is finite.
\end{proof}

\begin{prop} An element $a\in \Bbb A_{K}^{fin}$ is von Neumann regular if and only if 
$[[a]]_+^{fin}$ has finite support.\end{prop}
\begin{proof} The proof is similar to Proposition \ref{v-n-reg}. 
Consider $a$ with $[[a]]_+^{fin}$ of finite support. Let 
$e=[[x\neq 0]]^{fin}(a)$ and define $c$ by $c(v)=a(v)^{-1}$ if $e(v)=1$ where $v\in V_K^{f}$ 
and argue similarly as in proof of Proposition \ref{v-n-reg}. 

For the converse, argue similar to proof of Proposition \ref{v-n-reg} replacing $supp(a)$ by $[[a]]_+^{fin}$.
\end{proof}

\begin{cor} $\A_K^{fin}$ is not von Neumann regular.
 
\end{cor}
\begin{proof} Immediate.
\end{proof}

Note that this is in contrast to the case of direct products of fields which are von Neumann regular.

\begin{cor} An idempotent $e\in \A_K$ (resp. $e\in \A_{K}^{fin}$) 
has finite support if and only if for all $a\in \Bbb A_K$ 
(resp. $a\in \A_{K}^{fin}$), $ae$ is von Neumann regular.
\end{cor}

\begin{proof} We shall only prove the statements for $\A_K$ since the case of finite adeles $\A_{K}^{fin}$ is completely 
analogous. The left to right direction is clear. Conversely, assume $ae$ has is von Neumann regular for all $a$. 
Suppose $e$ has infinite support. Define $a\in \A_K$ by $a(v)=\pi_v$ if $e(v)=1$ and $a(v)=0$ otherwise. 
Then $a=ae$ and $[[a]]_+=[[ae]]_+=e$ has infinite support, contradiction.\end{proof}

The statement that "$x$ is von Neumann regular" can be expressed by an $\exists\forall\exists$-formula in the variable $x$ 
in the language of rings. This yields an $\forall\exists\forall\exists$-formula defining the ideal of finite support idempotents in $\A_K$ (and in 
$\A_K^{fin}$ uniformly for all 
number fieds $K$. 

\begin{prop}\label{prop-v-n-r} Let $a\in \Bbb A_K$. Then $(a) \subseteq ([[a]])$. Moreover, $(a)=([[a]])$ if and only if 
$a$ is von Neumann regular.\end{prop}

\begin{proof} Immediate.\end{proof}

\section{Model theory of Boolean algebras}\label{sec-bool}

In this section we state the results from the model theory of Boolean algebras from \cite{DM-bool} that we need. We let 
$\frak{L_0}$ be the usual language for Boolean algebras with signature $\{\vee,\wedge,^{-},0,1\}$, and let $\frak{L}$ an 
enrichment of it. While some of our results will hold for an arbitrary enrichment $\frak{L}$, the main applications concern the following enrichments:

{\it (i):} The enrichment of $\frak{L}_0$ by unary predicates $C_j(x)$ expressing that there are at least $j$ distinct atoms below $x$, and a predicate $Fin(x)$ for the ideal of finite sets. We shall denote this language by $\frak{L}^{fin}$. Let $T^{fin}$ denote the theory of infinite atomic Boolean algebras in the language $\frak{L}^{fin}$ with defining axioms for $C_j$ and $Fin$ (cf.~\cite{DM-bool} for details). The following result was proved by Feferman-Vaught \cite{FV}, and a new proof was given in \cite{DM-bool}.

\begin{thm}\label{bool1}\cite{FV,DM-bool} The theory $T^{fin}$ in the language $\frak{L}^{fin}$ is complete, decidable and 
has quantifier elimination.\end{thm}

{\it (ii):} The enrichment $\frak{L}^{fin,res}$ of $\frak{L}^{fin}$ by unary predicates $Res(n,r)(x)$ for $n,r\in \Z$, $n>0$, with the intended interpretation, in $Powerset(I)$, that $Fin(x)$ holds and the cardinality of $x$ is congruent to $r$ modulo $n$. Let $T^{fin,res}$ denote the theory of infinite atomic Boolean algebras in the language $\frak{L}^{fin,res}$ with the defining axioms for the predicates 
$C_j(x),~Fin(x)$ and $Res(n,r)(x)$ (cf. \cite{DM-bool} for details). The following theorem was proved in \cite{DM-bool}.

\begin{thm}\label{bool2}\cite{DM-bool} The theory $T^{fin,res}$ in the language $\frak{L}^{fin,res}$ is complete, decidable, and has 
quantifier elimination.\end{thm}

\section{Model theory of restricted products and direct products}\label{sec-rest}

In this section we state the results from \cite{DM-supp} that we will be using. These results generalize the work of 
Feferman and Vaught in \cite{FV} on quantifier elimination in a language for restricted products relative to 
the theory of the factors and the theory of the Boolean algebra of the index set in given languages.

Let $L$ be a 1-sorted or many-sorted first-order language. We assume that $L$ has the equality symbol $=$ of various sorts, 
and may have relation symbols and 
function-symbols. For accounts of many-sorted logic see 
\cite{KK,feferman,pillay-book,MR}. We do not assume as in 
\cite{pillay-book} that the sorts are disjoint. The well-formed equality statements demand that the terms involved be of the same sort.

Let $\mathcal{L}_0$ be the language for Boolean algebras $\{0,1,\wedge,\vee,\bar{} \ \}$. Let 
$\mathcal{L}$ be {\it any} extension of $\mathcal{L}_0$. Let 
$I$ be an index set, with associated Boolean algebra $\B:=Powerset(I)$ (which denotes 
the powerset of $I$). We denote by $\B_{\mathcal{L}}$ the $\mathcal{L}$-structure on $\B$ with the usual 
interpretations for $\{0,1,\wedge,\vee,\bar{} \ \}$.

Given an $L$-structure $M$, and a sort $\sigma$, let $\sigma(M)$ denote the $\sigma$-sort of $M$. 
Consider a family $\{M_i: i\in I\}$ of $L$-structures with product $\Pi=\prod_{i\in I} M_i$. 
The product 
$$\prod_{i\in I} Sort_{\sigma}(M_i)$$
is defined to be the $\sigma$-sort of the product $\Pi$. The elements are functions 
$f_{\sigma}$ on $I$ satisfying 
$$f_{\sigma}(i) \in Sort_{\sigma}(M_i)$$
for all $i$. Hence the product $\Pi$ is $L$-sorted. 

We shall write 
$$\bar f_{\sigma_1},\dots,\bar f_{\sigma_j},\dots$$
for tuples of elements of sorts 
$\sigma_1,\dots,\sigma_j,\dots$ respectively; and 
$$\bar x_{\sigma_1},\dots,\bar x_{\sigma_j},\dots$$
for tuples of $L$-variables of sorts 
$\sigma_1,\dots,\sigma_j,\dots$ respectively.

Let $\tau$ be a function-symbol of sort 
$$\sigma_1\times \dots \times \sigma_r \rightarrow \sigma.$$ 
Then the interpretation of $\tau$ in $\Pi$ is given by 
$$\tau^{(\Pi)}(\bar f_{\sigma_1},\dots,\bar f_{\sigma_r})(i)=
\tau^{(M_i)}(\bar f_{\sigma_1}(i),\dots,\bar f_{\sigma_r}(i)).$$
Given an $L$-formula $\Phi(\bar w_{\sigma_1},\dots,\bar w_{\sigma_r})$, the Boolean value is defined as 
$$[[\Phi(\bar f_{\sigma_1},\dots,\bar f_{\sigma_r})]]=
\{i: M_i \models \Phi(\bar f_{\sigma_1}(i),\dots,\bar f_{\sigma_r}(i))\}.$$

The interpretation of a relation symbol $R$ of sort 
$$\sigma_1\times \dots \times \sigma_r$$
is given by 
$$R^{(\Pi)}(\bar f_{\sigma_1},\dots,\bar f_{\sigma_r}) \Leftrightarrow 
[[R(\bar f_{\sigma_1},\dots,\bar f_{\sigma_r})]]=1.$$
This defines a product $L$-structure on $\Pi$, extending with the 1-sorted version.

We will write $z_1,\dots,z_j,\dots$ for variables of the language $\cL$. 

We define new relations on the product $\Pi$. Let $\Psi(z_1,\dots,z_m)$ be an $\mathcal{L}$-formula, and $\Phi_1,\dots,\Phi_m$ be $L$-formulas in a common set of 
variables $\bar x_{\sigma_1},\dots,\bar x_{\sigma_s}$ of sorts $\sigma_1,\dots,\sigma_s$ respectively. We define the relation 
$$\Psi \circ < \Phi_1,\dots,\Phi_m>$$ as 
$$\Pi\models \Psi \circ<\Phi_1,\dots, \Phi_m>(\bar f_{\sigma_1},\dots,\bar f_{\sigma_s}) \Leftrightarrow$$ 
$$\B_{\mathcal{L}}\models \Psi([[\Phi_1(\bar f_{\sigma_1},\dots,\bar f_{\sigma_s}),\dots,
[[\Phi_r(\bar f_{\sigma_1},\dots,\bar f_{\sigma_s})]]).$$
We extend $L$ by adding a new relation symbol, of appropriate arity, for each of the above. In this way we get 
$L(\B_{\mathcal{L}})$, and $\Pi$ has been given an $L(\B_{\mathcal{L}})$-structure. This is a generalization to the 
many-sorted case of the language of generalized products in \cite{FV}.

Suppose $M$ and $N$ are $L$-structures. We put $N_{\sigma}=Sort_{\sigma}(N)$ for every sort $\sigma$. 
An $L$-morphism $F:N\rightarrow M$ is by definition a collection of maps 
$$F_{\sigma}: N_{\sigma} \rightarrow M_{\sigma},$$
where $\sigma$ ranges over the sorts, such that for any relation symbol $R$ of sort 
$$\sigma_1\times \dots \times \sigma_k$$
we have, 
$$N_{\sigma_1}\times \dots \times N_{\sigma_k} \models R(\bar f_1,\dots,\bar f_k)\Leftrightarrow 
M_{\sigma_1}\times \dots \times M_{\sigma_k} \models R(F_{\sigma_1}(\bar f_1),\dots,F_{\sigma_k}(\bar f_k)),$$
and for any function symbol $G$ of sort
$$\sigma_1\times \dots \times \sigma_k \rightarrow \sigma$$
we have 
$$G(F_{\sigma_1}(\bar f_1),\dots,F_{\sigma_k}(\bar f_k))=F_{\sigma}(G(\bar f_1,\dots,\bar f_k)),$$
where $\bar f_1,\dots,\bar f_k$ denote tuples of elements of sorts $\sigma_1,\dots,\sigma_k$ respectively.

We remark that our convention that we have equality as a binary relation on each sort forces each 
$F_{\sigma}$ to be injective.

If each $N_{\sigma} \subseteq M_{\sigma}$, and the identity maps are $L$-morphisms, then we say 
{\it $N$ is an $L$-substructure of $M$}.

We now define a many-sorted generalization of Feferman-Vaught's notion of a weak product of structures.  
Assume that for each sort $\sigma$ we have a formula $\Phi_{\sigma}(x_{\sigma})$ in a 
single free variable $x_{\sigma}$ of 
sort $\sigma$, and assume that for each $\sigma$ for all $i$ the sets 
$$S_{\sigma,i}=\{x\in Sort_{\sigma}(M_i): M_i\models \Phi_{\sigma}(x)\}$$
are $L$-substructures of $M_i$. In particular, for any function symbol $F$ of sort 
$$\sigma \rightarrow \tau$$
and any $a\in S_{\sigma}(i)$ we have that 
$$F(a)\in S_{\tau}(i),$$
for all $i$.

We assume that $\frak{L}$ contains a predicate $Fin(x)$ for the finite sets. 

We define $\Pi^{(\Phi_{\sigma})}$ (also denoted $\prod_{i\in I}^{(\Phi_{\sigma})} M_i$), as the 
$L(\B_{\mathcal{L}})$-substructure of $\Pi$ whose sort $\sigma$ consists of the 
$f_{\sigma}\in \prod_{i\in I} S_{\sigma}(M_i)$ 
such that 
$$Fin([[\neg \Phi_{\sigma}(f_{\sigma})]])$$
holds. We call it the restricted product of $M_i$ with respect to the formulas $\Phi_{\sigma}(x)$. 

$\Pi^{(\Phi_{\sigma})}$ is $L$-sorted, namely, given $\sigma$, a sort of $L$, 
the $\sigma$-sort of $\Pi^{(\Phi_{\sigma})}$ is the set of all $f_{\sigma}\in \prod_{i\in I} S_{\sigma}(M_i)$ such that 
$$Fin([[\neg \Phi_{\sigma}(f_{\sigma})]])$$
holds. 

Note that if $F$ is a function symbol of sort 
$$\sigma \rightarrow \tau,$$
and $a$ is in the $\sigma$-sort of $\Pi^{(\Phi_{\sigma})}$, then since the sets $S_{\sigma}(i)$ are 
$L$-substructures of $M_i$ for all 
$i$, we see that 
$$Fin([[\neg \Phi_{\tau}(F(f_{\sigma}))]])$$
holds. Hence $F(a)$ is in $Sort_{\tau}(\Pi^{(\Phi_{\sigma})})$. 
This shows that $\Pi^{(\Phi_{\sigma})}$ is a substructure of $\Pi$. It is $L(\B_{\mathcal{L}})$-definable.

The following result gives a quantifier-elimination for restricted products in the language 
$L(\B_{\mathcal{L}})$.

\begin{thm}[\cite{DM-supp}]\label{restricted-qe} 
Let $L$ and $\mathcal{L}$ be as before. 
For any $L(\B_{\mathcal{L}})$-formula $\Psi(x_1,\dots,x_n)$, one can effectively construct 
$L$-formulas $\Psi_1(x_1,\dots,x_n),\dots,\Psi_m(x_1,\dots,x_n)$ 
and an $\mathcal{L}$-formula $\Theta(X_1,\dots,X_m)$ 
such that given any family $\{M_i: i\in I\}$ of $L$-structures 
and any Boolean ${\mathcal{L}}$-structure on $Powerset(I)$ denoted $\B_{\mathcal{L}}$, and any 
formula $\Phi(x)$ from $L$, for any $a_1,\dots,a_n\in \prod^{(\Phi)}_{i\in I} M_i$ we have, 
$$\prod^{(\Phi)}_{i\in I} M_i\models \Psi(a_1,\dots,a_n)$$ if and only if 
$$\B_{\mathcal{L}}\models \Theta([[\Psi_1(a_1,\dots,a_n)]],\dots,[[\Psi_m(a_1,\dots,a_n)]]).$$
\end{thm}

Note that if the index set is finite we may refine this somewhat, giving a many-sorted version of 
results of Mostowski which predate \cite{FV}. 
This will be useful to us when we write the adeles as a product of the finite adeles and a 
finite product of archimedean completions.

\begin{thm}\label{fv-fin-prod} Consider a finite  index set $I=\{1,\dots,s\}$. 
Let $\psi(x_1,\dots,x_n)$ be an $\cL$-formula. Then there are finitely many 
$s$-tuples of formulas 
$$(\psi_1(x_1,\dots,x_n),\dots,\psi_t(x_1,\dots,x_n))$$
for some $t\in \N$, 
and elements $S_1,\dots,S_k$, for some $k\in \N$, where each $S_j$
is in $Powerset(I)^{t}$, 
such that for arbitrary $\cL$-structures $M_1,\dots,M_s$, and any 
$f_1,\dots,f_n$ in $M_1\times \dots \times M_s$
$$M_1\times \dots \times M_s\models \psi(f_1,\dots,f_n)$$
if and only if for some $j$ the sequence 
$$[[\psi_1(f_1,\dots,f_n)]],\dots,[[\psi_t(f_1,\dots,f_n)]]$$ is equal to $S_j$.

\end{thm}
\begin{proof} Immediate.\end{proof}

\begin{cor}\label{mostowski} Let $A\subset M_1 \times \dots \times M_s$ be an $\cL$-definable set.
Then $A$ is a finite union of rectangles
 $ B_1 \times \dots \times B_s$ , where $B_i$ is a definable subset of $M_i$.
\end{cor}
\begin{proof} Immediate.
\end{proof}

One should note that we do not claim any converse of the last corollary, in general, though
it will be true in many ring-theoretic settings.

\section{Quantifier elimination and definable sets}\label{sec-qe}

In this section we shall first prove a general quantifier elimination theorem for the adeles. We shall then give instances of the result.

Weispfenning \cite{Weisp2} already exploited the fact that the adele construction is closely connected to the generalized products of Feferman and Vaught \cite{FV}. The essential point is that $\Bbb A_K$ is built up using 

i) A family $\{K_v: v \in V_K\}$ of structures whose $L$-theory is well-understood, for some language $L$, 

ii) An enrichment $\cL$ of the Boolean algebra $Powerset(V_K)$ containing the predicates $Fin$ and $Card_n(x)$.  
Thus $Powerset(V_K)$ becomes an $\cL$-structure denoted $\Bbb P(I)^+$. 

ii) Some restriction on the product $\prod_{v \in V_K} K_v$, defined in terms of the enriched Boolean algebra $\Bbb P(I)^+$.

Here the restriction in (ii) is that $Fin(\{v: \neg \Phi_{val}(x(v))\})$ where 
$\Phi_{val}(x)$ expresses that $x\in \mathcal{O}_v$. 

\

\subsection{\it The general quantifier elimination theorem}

\

\begin{thm}\label{th-qe} Let $K$ be a number field. Let $\frak{L}$ be an extension of the language of 
Boolean algebras such that an enrichment of the Powerset, $\P(I)^{+}$, 
has quantifier elimination in $\cL$. Let $L$ be a many-sorted 
language such that the non-archimedean completions $K_v$, where $v\in V_K^f$, 
admit uniform quantifier elimination in a sort $\sigma$ relative to other sorts. Then 
for any $L$-formula $\psi(x_1,\dots,x_n)$ there exist $L$-formulas 
$$\psi_1(x_1,\dots,x_n),\dots,\psi_m(x_1,\dots,x_n)$$ 
which are quantifier-free in the sort $\sigma$, and 
a quantifier-free Boolean $\cL$-formula $\theta(X_1,\dots,X_m)$, all of which are effectively computable from $\psi_1,\dots,\psi_m$, such that for any 
$a_1,\dots,a_n\in \A_K^{fin}$,
$$\Bbb A_K^{fin}\models \Psi(a_1,\dots,a_n) \Leftrightarrow$$
$$\B_{K}\models \theta([[\psi_1(a_1,\dots,a_n)]],\dots,[[\psi_m(a_1,\dots,a_n)]])).$$
\end{thm}
\begin{proof} Since $\A_{K}^{fin}$ is the restricted direct product of $K_v$, $v\in V_K^f$, with respect to the formula 
$\Phi_{val}(x)$, we may apply Theorem \ref{restricted-qe} and get formulas 
$$\psi'_1(x_1,\dots,x_n),\dots,\psi'_m(x_1,\dots,x_n)$$ and a Boolean formula $\theta(X_1,\dots,X_m)$ 
such that for any $a_1,\dots,a_n\in \A_K^{fin}$,

$$ \A_K^{fin}\models \Psi(a_1,\dots,a_n)$$ if and only if 
$$\Bbb P(I)^+ \models \theta([[\psi'_1(a_1,\dots,a_n)]],\dots,
[[\psi'_m(a_1,\dots,a_n)]]).$$
Now use the quantifier elimination in $\mathcal{L}$.
\end{proof}


We now give examples of languages in which $K_v$, for all non-archimedean $v$, have uniform quantifier elimination 
in a certain sort. Each of these languages can be used in Theorem \ref{th-qe} in place of $L$. 

\

i) {\it Enrichments of the ring language and the Macintyre language}

\

This is a one-sorted language. Expand the language of rings $\cL_{rings}$ by the Macintyre power 
predicates $P_n(x)$ expressing that $x$ is an $n$th 
power, for all $n$. For any $n$, add constants to be interpreted in $K_v$ as coset representatives of the group of non-zero $n$th powers $(K_v^*)^n$ in $K_v$. 
Note that the index of $(K_v^*)^n$ in $K_v$ is bounded independently of $v$, 
so we have a finite set of constants independently of $v$. 
Add constants for an $\F_p$-basis of $\mathcal{O}_{K_v}/(p)$, for all $p$ (the number of such basis is the 
$p$-rank in the sense of 
Prestel-Roquette and it is bounded and depends only on the dimension of the number field $K$ over $\Q$). 
For all $m\geq 2$, add predicates $Sol_m(x_1,\dots,x_m)$ interpreted by 
$$\bigwedge_{1\leq i\leq m} v(x_i)\geq 0 \wedge \exists y (v(y)\geq 0 \wedge v(y^m+x_1y^{m-1}+\dots+x_m)>0).$$ 
Add s a symbol for uniformizer, and a binary relation $D(x,y)$ defined by $v(x)\leq v(y)$. 

This language was defined by Belair \cite{Belair}, where he proved that the fields $\Q_p$, for all $p$, 
admit uniform elimination of quantifiers. 
A related many sorted language with sorts for the field and residue rings $R_k:=\mathcal{O}_K/\cM_K^k$ with the ring language, symbols for the reside maps into residue rings, where ($k\geq 1$), (note that $R_1$ is the residue field), and a sort for the value group with the 
language of ordered abelian groups, was defined by Weispfenning \cite{Weisp2}, where he proved a relative quantifier elimination for the field sort relative to residue rings $R_k$ and value group.

In these languages, the non-archimedean completions $K_v$ of a fixed number field $K$, where $v\in V_K^f$, have  elimination of quantifiers for almost all 
$v$ (namely of norm larger than some constant) for the field sort relative to the other sorts (in the many-sorted case). In the case of $K=\Q$, all $\Q_p$ have uniform quantifier elimination in Belair's language (cf. \cite{Belair}). Belair's results can also be 
deduced from the result of Weispfenning's \cite{Weisp2} cited above as follows: 
first get a quantifier elimination for the field sort 
relative to the residue ring sorts $R_K$. Then replace the formulas for the residue ring sorts in terms of the residue field for all $K_v$ which are unramified (which is the case for all but finitely many $v$), then one 
uses quantifier elimination in the residue fields $k_v$ for all but finitely many $v$ due to Kiefe \cite{kiefe} in the extension of the ring language with the predicates $Sol_k$. Then an argument of 
Belair \cite{Belair} extends this quantifer elimination to all $\Q_p$. Note that in this case there are coset representatives for $(\Q_p^*)^n$ which are integers for all $n$, so we do not need to name them by constants.

\

ii) {\it The language of Basarab-Kuhlmann}

\

This language has sorts $(K,K^*/1+\cM^n,\cO_K/\cM^m,\Gamma,v,\pi_n,\pi_n^*)$, for all $n,m\geq 1$, with the language of groups $\{.,^{-1},1\}$ for 
the sort $K^*/1+\cM^n$, the language of rings for the sorts $\cO_K/\cM^m$, the language of ordered abelian groups for the value group $\Gamma$, with symbols for the valuation and canonical projection maps $\pi_n: \cO_K \rightarrow \cO_K/\cM^m$ and $\pi_n^*: K^* \rightarrow K^*/1+\cM^n$ 
from the field sort into the other sorts. Equivalently, one can replace the maximal ideal $\cM$ by the ideal generated by $p$ (which differ if the ramification index is greater than $1$). 
By a theorem of Basarab \cite{basarab} and Kuhlmann \cite{kuhlmann}, the completions $K_v$ for almost all $v$, have uniform quantifier elimination in the field sort relative to other sorts for a given number field $K$.

\

iii) {\it The adelic Denef-Pas language with a product angular component map and product valuation}

\

The Denef-Pas language $\cL_{PD}$ has three sorts $(K,k,\Gamma,ac,v)$ with the ring language for the field $K$ and the residue field $k$, the language of ordered abelian groups for the value group $\Gamma$, and symbols for the valuation and the 
angular component map modulo $\cM$ defined by $ac(x)=res(x\pi^{-v(x)})$, where $\pi$ is a uniformizer and $res$ the residue map into the residue field $k$ if $x\neq 0$, and $ac(0)=0$. $ac$ is defined up to a choice of $\pi$ and is not definable, but can be defined from a cross section (a section to the valuation map), and exists in local fields and saturated valued fields. It is axiomatized on non-zero elements as being multiplicative and coinciding with the residue map on units. 
By a theorem of Pas \cite{pas}, given $K$, the completions $K_v$, for almost all $v$, have uniform quantifier elimination in $\cL_{PD}$.

Using $\cL_{PD}$ we get a language for the finite adeles $\A_K^{fin}$ as in Section \ref{sec-rest}. This language, that we call adelic Denef-Pas, has three sorts
$$(\A_K^{fin},\prod_{v\in V_K} \Gamma_v,\prod_{v\in V_K} k_v,ac^*,v^*),$$ with the ring language for 
$\A_K^{fin}$, the language of lattice-ordered abelian groups for the lattice ordered group $\prod_{v\in VK} \Gamma_v$ which is defined as the direct product of the value groups $\Gamma_v$ (which are all $\Z$), and the language of rings for the direct 
product of the residue fields $k_v$ (of $K_v$). The adelic angular component $ac^*$ and product valuation both defined on $\A_K^{fin}$ are defined by 
$ac^*(a)=(ac(a(v)))_v$ and $v^*(a)=(v(a(v))_v$, for $a\in \A_K$. In \cite{DM-supp}, we studied the model theory of the adeles with the product valuation map. The results hold also for the adelic Denef-Pas language. 

\

By Theorem \ref{th-qe}, there is a quantifier elimination for the finite adeles $\A_K^{fin}$ and the adeles $\A_K$ taking $L$ to be 
each of the above languages, and $\cL$ appropriately chosen as in the theorem.

\begin{thm}\label{th-def-sets} Let $X$ be a definable subset of 
$(\Bbb A_K^{fin})^m$, in the ring language, where $m\geq 1$. Then $X$ is a finite Boolean 
combination of sets of the following types:
\begin{enumerate}
\item (Type I) $\{(x_1,\dots,x_m)\in \Bbb A_K^m: C_j([[\varphi(x_1,\dots x_n)]])\}$, where $j\geq 1$, 
\item (Type II) $\{(x_1,\dots,x_m)\in \Bbb A_K^m: Fin([[\psi(x_1,\dots,x_n)]])\}$.
\end{enumerate} where $\varphi(\bar x)$ and $\psi(\bar x)$ can be chosen to be formulas that are quantifier free in a sort $\sigma$ from any language $L$ such that the completions $K_v$ have uniform quantifier elimination in sort $\sigma$ of $L$ relative to other sorts. In case $L$ is one-sorted, the fomulas $\theta,\psi$ can be chosen to be quantifier free in $L$. This holds for the languages of Belair, Weispfenning, Basarab-Kuhlmann, and adelic Denef-Pas.
\end{thm}
\begin{proof} We shall use the notation in Theorem \ref{th-qe}. Suppose that $X$ is defined by a formula $\phi(\bar x)$. Applying Theorem \ref{th-qe} with $\mathcal{L}$ the language $\mathcal{L}^{fin}$ in Section \ref{sec-bool} and $\Bbb P(I)^+$ the corresponding enrichment of the power set Boolean algebra to $\cL$ (i.e. with predicates $C_j$ and $Fin$), we get $L$-formulas $\psi_1,\dots,\psi_m$ and an $\cL$-formula $\theta(X_1,\dots,X_m)$ 
such that for any 
$a_1,\dots,a_n\in \A_K^{fin}$,
$$\Bbb A_K^{fin}\models \Psi(a_1,\dots,a_n) \Leftrightarrow$$
$$\B_{K}\models \theta([[\psi_1(a_1,\dots,a_n)]],\dots,[[\psi_m(a_1,\dots,a_n)]])).$$
Applying Theorem \ref{bool1}, we deduce that the formula $\theta(X_1,\dots,X_m)$ can be chosen to be 
quantifier free in $\cL$. It is thus a finite Boolean combination of formulas of the form 
$$C_j(\beta(X_1,\dots,X_m))$$ and 
$$Fin(\gamma(X_1,\dots,X_m)),$$
where $\beta$ and $\gamma$ 
are terms $\cL$ not involving $C_j$ and $Fin$. Now using the correspondence between subsets of the set of 
valuations and idempotents in $\A_K^{fin}$ (see Section \ref{sec-id}) we deduce that 
for $a_1,\dots,a_n\in \Bbb A_K^{fin}$, 
$$\Bbb P(I)^+\models \theta([[\psi_1(\bar a)]],\dots,[[\psi_m(\bar a)]])$$ 
if and only if
$$\B_K\models \theta([[\psi_1(\bar a)]],\dots,[[\psi_m(\bar a)]]),$$
if and only if
$$\Bbb A_K^{fin}\models \theta([[\psi_1(\bar a)]],\dots,[[\psi_m(\bar a)]]).$$
Now we apply the quantifier elimination in $L$ that is uniform in $v$ to deduce that each Boolean value 
$[[\psi_j(\bar x)]]$ equals the Boolean value $[[\psi'(\bar x)]]$ where $\psi'(\bar x)$ is quantifier-free (in relevant sort $\sigma)$.
This completes the proof.
\end{proof}

\begin{remark} Note that sets of the form 
$$\{(x_1,\dots,x_m)\in \Bbb A_K^m: ([[\phi(x_1,\dots,x_n)]]=1)\}$$
are a special case of sets of Type I. Indeed, in an atomic Boolean algebra one has
$$\beta=0 \Leftrightarrow \neg C_1(\beta)$$
for any Boolean term $\beta(\bar X)$. These sets generalize the adele space (or adelization) of a variety $V$ 
given by a system of polynomial equations $f_{\alpha}(x_1,\dots,x_n)$ over a number field $K$ 
defined by $\{\bar x \in \A_K: f_{\alpha}(\bar x)=0\}$.

\end{remark}

\begin{cor} A definable subset $X$ of 
$\Bbb A_{K}^m$ is a finite union of rectangles $B\times W\times S$ where $B$ is a finite union of 
rectangles $\prod_{v~\mathrm{real}} B_v$, where $B_v$ a is quantifier-free definable subset of $\R^{m_v}$ in the language of ordered fields, $W$ is a finite union of rectangles 
$\prod_{v~\rm{complex}} W_v$, where $W_v$ a quantifier-free definable subset of $\C^{m_v}$ in the language of rings, 
and $S$ is definable in $\A_K^{fin}$ (hence has the form given in Theorem \ref{th-def-sets}.
\end{cor}
\begin{proof} Immediate by Corollary \ref{mostowski}.\end{proof}

\section{Definable subsets of the adeles}\label{sec-def-sets}

In this section we show that definable subsets of the adeles in the ring language are measurable but not necessarily finite unions of locally closed sets 
in contrast to the case of the completions $K_v$, $v\in V_K^f$, where definable sets in $K_v^m$ are finite unions of locally closed sets in each of the languages (i)-(iii) in Section 
\ref{sec-qe}, for any $m$. This follows from quantifier elimination. This property also holds in 
the real and complex fields by Tarski's quantifier elimination (cf \cite{KK}).


\begin{thm}\label{th-def-meas} The definable subsets of $\Bbb A_K^m$ in the language of rings are measurable for any $m\geq 1$.\end{thm}
\begin{proof} Since the definable sets in 
$\R$ and $\C$ are measurable (by Tarski's quantifier elimination theorems, cf.~\cite{KK}), by Corollary \ref{mostowski}, it suffices to prove that the definable subsets of 
the finite adeles $\A_K^{fin}$ are measurable. For this we apply Theorem \ref{th-def-sets}. For Type I sets, 
let us define unary predicates $\mathcal{C}_j(x)$ expressing that 
there are exactly $j$ atoms below $x$ (where $j\geq 1$). 
In an atomic Boolean algebra, the predicates $C_j$ can 
be defined in terms of the predicates $\mathcal{C}_j$. Thus in $\B_K$ one has 
$\{\bar x: C_j([[\Psi(\bar x]])\}=\bigcup_{k\geq j} 
\{\bar x: \mathcal{C}_j([[\Psi(\bar x)]])\}$. So it suffices to show that the 
sets defined by the $\mathcal{C}_j$ are measurable. To see this note that 
$$\mathcal{C}_j([[\Psi(\bar x)]])=\bigcup_{v_{i_1},\dots,v_{i_j}\in \mathrm{S(K)}} 
(\bigcap_{1\leq t\leq j} 
(\{\bar x: K_{v_{i_t}}\models \Psi(\bar x(v_{i_t}))\}$$
$$\cap \bigcap_{w\notin \{v_{i_1},\dots,v_{i_j}\}} 
\{\bar x: K_w\models \neg \Psi(\bar x(w)))),$$
where the union is over all $j$-tuples of distinct normalized non-archimedean valuations. 
By the earlier remark on Belair's quantifier elimination theorem \cite{Belair} and measurability, every definable set in  
$K_v$ is measurable. Hence $X$ is also 
measurable. 

Now consider the Type II case. We have
$$Fin([[\Psi(\bar x)]])=
\bigcup_{\mathcal{F}} (\bigcap_{v\in F} \{\bar x: K_v\models \Psi(\bar x(v))\}
\cap \bigcap_{w\neq v} \{\bar x: K_w\models \neg \Psi(\bar x(w))\}),$$
where $\mathcal{F}$ ranges over all the finite subsets of the set of normalized non-archimedean valuations $V_K^{f}$.
So arguing similarly as in the Type I case we deduce that $Fin([[\Psi(\bar x)]])$ is measurable.\end{proof}

Strengthening the Boolean language to be $\cL^{fin,res}$ we still get measurability.
\begin{thm} Let $X$ be a definable subset of $(\A_K^{fin})^m$ in the language $\cL^{fin,res}(L)$, where $L$ is any of the languages (i)-(iii) in Section \ref{sec-qe}, and $m\geq 1$. Then $X$ is measurable.\end{thm}
\begin{proof} The proof is similar to the proof of Theorem \ref{th-def-meas} but instead of considering finite subsets of $V_K^f$ or subsets of size $j$, we consider sets of 
satisfying $Res(n,r)(x)$, for all $n,r$.
\end{proof}
Note that $\cL^{fin,res}(L)$-definability in $\A_K$ is in general stronger than $\cL_{rings}$-definability in $\A_K$.


Given a topological space 
$X$ and a subset $A\subseteq X$, we denote the frontier of $A$ by $fr(A):=cl(A)\setminus A$. The locally closed part of $A$ is defined to be $A\setminus cl(fr(A))$ and denoted by $lc(A)$. Let $k\geq 1$. 
We put $A^{(0)}=A$, and $A^{(k+1)}=A^{(k)}\setminus 
lc(A^{(k)})$. We also 
put $fr^{(0)}(A)=A$ and $fr^{(k+1)}(A)=fr(fr^{(k}(A))$. Note that $A^{(k)}=fr^{(2k)}(A)$.

The following was proved by Dougherty and Miller \cite{D-miller}.
\begin{fact}\label{miller} Let $X$ be a topological space and $A\subseteq X$.\begin{itemize}
\item Let $k\geq 1$. Then $A$ is a union of $k$ locally closed sets if and only if $A^{(k)}=0$.
\item $A$ is a Boolean combination of open sets if and only if $A^{(k)}=0$ for some $k\geq 1$.
\end{itemize}\end{fact}
\begin{proof} See \cite{DM}, page 1348.\end{proof}

We can now prove.
\begin{thm} Let $X$ be the definable set $\{a\in \A_K: Fin([[a\neq a^2]])\}$. Then $X$ is not a finite Boolean combination of open  sets in $\A_K$.\end{thm}
\begin{proof}
It is easy to see that
$$fr(fr(A))=A.$$
Thus $fr^{(k)}(A)=A$ for all $k\geq 1$. Hence $A^{(k)}=A$ for all $k\geq 1$. Thus by Fact \ref{miller}, $A$ is not a union of $k$-many locally closed sets, for any $k$.\end{proof}

By Theorem \ref{th-def-sets}, the definable subsets of the adeles in the ring language 
are at the lowest level of the Borel hierarchy: 
sets of the form $\{\bar f: [[\phi(\bar f)]]=1\}$, where $\phi$ is a ring formula are finite unions of locally closed sets, 
sets of the form $\{\bar f: Fin([[\phi(\bar f)]])\}$ are countable unions of locally closed sets, and set of the form 
$\{\bar f: \neg Fin([[\phi(\bar f)]])\}$ are countable intersections of locally closed sets.

\subsection{\it Measures of definable sets}

\

\

Measures of definable subsets of $\A_K^m$, for $m\geq 1$, are related to numbers of arithmetical significance. A general description of them seems out of reach at present. However, 
in this section we shall give examples of measures of definable sets related to values of zeta and $L$-functions.

\

{\it i) Zeta functions in the language of rings} 

\

Let $\psi(x)$ be the formula $0\leq v(x) \leq 1$ (note that this statement has an $\cL_{rings}$-definition). 
Then the set in $\Bbb Q_p$ 
defined by $\psi(x)$ has measure $1-p^{-2}$, and the measure of the set 
defined by $\psi(x)$ in the finite adeles $\A_{\Bbb Q}^{fin}$ is the Euler product 
$$\prod_{p}(1-p^{-2})=(\zeta(2))^{-1}.$$
Similarly, the definable set $\{x: v(x)=0\}$ in $\Bbb Q_p$ has measure 
$1-p^{-1}$ and the measure of the corresponding definable set in $\A_{\Q}^{fin}$ is 
$\prod_{p}(1-p^{-1})=0$. Thus reciprocals of zeta values at positive integers are thus measures of definable sets in the finite adeles. 

Now let $\theta(x):=\psi(x) \wedge \exists w (w^2=-1)$. Then 
the measure of the set defined by $\psi(x)$ in $\Bbb A_{\Bbb Q}^{fin}$ is 
$\prod_{p\equiv 1(4)}(1-p^{-2})$ which is related to the zeta function of the 
quadratic extension $\Bbb Q(i)$. 

In general, Euler products of the form $\prod_{p\in S} (1-p^{-n})$, where $S$ is a set of primes $p$ where some sentence $\sigma$ of the language of rings holds in $\Bbb F_p$ will
be the measure of a definable set in the finite adeles $\A_K^{fin}$. Similarly $1/\zeta(s)$ can be shown to be an 
$\cL_{rings}$ -"definable integral" in $\A_{\Q}^{fin}$ (namely, an integral over a definable set in $\A_K$ of a function of the form $|f(\bar x)|^s$, where $f$ is a definable function in $\A_K$). 

\

{\it ii) Zeta functions in the adelic Denef-Pas language} 

\

While we saw that $1/\zeta(n)$ can be written as a measure of a definable set in $\A_{\Q}^{fin}$ in the ring language, we can show that $\zeta(n)$ itself is the measure of a definable set in $\A_{\Q}^{fin}$ in the adelic Denef-Pas language. 
Indeed, we show that $(1-p^{-n})^{-1}$ can be written as a measure of a definable set in $\Q_p$ in the Denef-Pas language, say defined by a formula $\psi(x)$. Then the set in $\A_{\Q}^{fin}$ defined by $[[\psi(x)]]^{na}=1$ will have measure equal to 
$\prod_p (1-p^{-n})^{-1}$. Now $1-p^{-n}$ equals $1+p^{-n}+p^{-2n}+\dots$. Let $\mu$ denote the normalized additive  Haar measure on $\Q_p$. Note that $1$ is the measure of the set
$$X:=\{x\in \Q_p: v(x)=-1 \wedge ac(x)=1\}.$$
To see this note that 
$$\mu(X)=\mu(\{x=p^{-1}u: v(x)=-1 \wedge ac(x)=1\}$$
$$=p\mu(\{u: v(u)=0 \wedge ac(u)=1\}=pp^{-1}=1.$$
Note that this set lies in the complement of $\Z_p$ in $\Q_p$. We now show that the sum is a measure of a definable set in 
$\Z_p$. For $n\geq 1$, let 
$$X_n:=\{x\in \Q_p: v(x)=n-1 \wedge ac(x)=1\}.$$
Then 
$$\mu(X_n)=\mu(\{x=p^{n-1}u: v(x)=n-1 \wedge ac(x)=1\}$$
$$=p^{-n+1}\mu(\{u: v(u)=0 \wedge ac(u)=1\}=p^{-n+1} p^{-1}=p^{-n}.$$
So consider the set
$$Y:=\{x\in \Q_p: v(x)\geq 0 \wedge v(x)\equiv 1 (mod~ k)\}.$$
Then this set is definable in $\Q_p$ and has measure $p^{-n}+p^{-2n}+\dots$. The disjoint union of $Y$ and $X$ is the intended set. 

Similarly, one can show that values of zeta functions $\zeta(s)$ are "definable" integrals in the Denef-Pas language for $\A_{\Q}
^{fin}$. 

If $f(\bar x)$ is a definable (in some setting) function and $X$ a definable subset of $\A_{\Q}^{fin}$, it is not always true that the adelic "definable" integrals $\int_X f(\bar x) d\mu$, where $d\mu$ is some measure on the adeles, have meromorphic continuation beyond their abscissa of convergence. For example, if we take the function $f(x,y,z)$ that at all places $p$ is defined by $|x|^{s+1}|y|^s|z|^s$, then arguing above one can show that
$$\int_X f(x,y,z) d\mu=\prod_p \left(1+\frac{p^{-1-s}}{(1-p^{-s})}\right)$$
which converges for $Re(s)>0$ but has $Re(s)=0$ as a natural boundary. 

More results on analytic properties of adelic definable integrals will appear in our work \cite{DM-ad2}. The question of which 
"definable" adelic integrals of a complex variable $s$ admit meromorphic continuation to the left of their abscissa of convergence is a challenging open problem. 

In Tate's thesis (cf.~ \cite{CF}), Tate normalizes the multiplicative measures $dx/|x|$ on the local fields $K_v^*$ by multiplying by $(1-q_v^{-1})$, where $q_v$ is the cardinality of the residue field of $K_v$, for all $v\in V_K^f$. 
These measures give a measure on the ideles 
$\Bbb I_K$ that gives the meromorphic continuation and functional equations of zeta functions in Tate's work. 
By the above, these normalization factors 
are measures of definable sets in Denef-Pas language, and thus the measures are "definable". The corresponding adelic and idelic measures are thus "definable" in the adelic Denef-Pas language. The same normalization factors are also used in the theory of  Tamagawa measures on adelic spaces of varieties.

\

\subsection{\it Definable sets of minimal idempotents in the adeles}

\

\

A natural question is what the definable subsets are of the set of minimal idempotents in the adeles. The following theorem provides an answer using the work by Ax \cite{ax}.

\begin{thm} Let $X$ be a definable subset of the set of minimal idempotents in $\A_K$. Then $X$ is in the Boolean algebra generated by finite sets of minimal idempotents 
and sets of minimal idempotents supported on sets of places of the form
$$\{v: k_v \models \sigma\}$$
where $k_v$ is the residue field of the local field $K_v$, and $\sigma$ is a sentence of the language of rings. Such a sentence can be chosen to be a Boolean 
combination of sentences of the form $\exists x f(x)=0$, where $f$ is a polynomial over $\Bbb Z$ in the single variable $x$. 
The above sets of places (if infinite) have the form
$$\{\frak p: Frob_{\frak p} \in C\}$$
where $C$ is a subset of $Gal(L/\Bbb Q)$ closed under conjugation, $L$ is a finite Galois extension of $\Bbb Q$, and $Frob_{\frak p}$ 
is the Frobenius conjugacy class.\end{thm}
\begin{proof} By Theorem \ref{th-def-sets}, $X$ is a Boolean combination of sets of the form
$$Fin([[\varphi(x)]])$$
and 
$$C_j([[\theta(x)]]).$$
Since $X$ is a set of minimal idempotents, it follows that the formulas 
$\varphi(x)$ or $\theta(x)$ can be replaced by sentences . Now we use the Ax-Kochen-Ershov theorem stating that given a sentence $\sigma$, 
there is a prime $p_0$ such that for any $p\geq p_0$, for any non-archimedean local field $K$ of residue characteristic at least $p_0$ with residue field $k$ and value group $\Gamma$, 
$$K\models \sigma \Leftrightarrow (k\models \rho \wedge \Gamma \models \Xi)$$
where $\rho$ is a Boolean combination of sentences from the language of rings, and $\Xi$ is a Boolean combination of sentences from the language of ordered abelian groups. This can be deduced from 
the residue zero characteristic case of the quantifier elimination in Basarab-Kuhlmann's language \cite{kuhlmann}.

Since the value groups of all the local fields under consideration are isomorphic, we see that $\sigma$ is equivalent in all Henselian valued fields of residue characteristic zero (hence in all Henselian fields of large residue characteristic) 
to a Boolean combination of residue field statements. Hence the Boolean values 
$$\{v: K_v \models \sigma\}$$
and
$$\{v: k_v \models \sigma\}$$
are the same except for a finite subset, where $k_v$ is the residue field of $K_v$. Now by the results of Ax \cite{ax}, these sets, if infinite, have the required description.\end{proof}

\

\subsection{\it An enrichment of the Boolean sort}

\

\

In Theorem \ref{th-qe}, if we take as $\mathcal{L}$ the language $\mathcal{L}^{fin,res}$ (and $L$ as before), then 
we get more expressive power than the language $\mathcal{L}^{fin}(L)$ that is used in Corollary 
\ref{th-def-sets}. 

\begin{thm} The $\mathcal{L}^{fin}(L)$-theory of $\A_K$ is decidable and has quantifier elimination.
\end{thm}
\begin{proof} Apply Theorem \ref{th-qe} and Theorem \ref{bool2}.\end{proof}

The $\mathcal{L}^{fin}(L)$-theory of the adeles is rich in connection to arithmetic and in particular reciprocity laws. An example of such 
a connection is the following. We expect many more such result to hold.

\begin{prop} The set of adeles $a\in \A_{\Q}$ such that $a(v)$ is a non-square at an even number of places $v$ is $\mathcal{L}^{fin,res}(L)$-definable in $\A_{\Q}$.\end{prop}
\begin{proof} Give $x,y \in \Q_v$, where $v\in \Bbb P \cup \{\infty\}$, where $\Bbb P$ denotes the set of primes, the Hilbert symbol is denoted as $(x,y)_p$. It is known that
$(x,y)_v=1$ if $x$ or $y$ is a square, and $(a,b)_v$ takes values in $\{-1,1\}$.

Now let $a,b \in \Q$. Then it is well-known that $(a,b)_v=1$ for all but finitely many $v$, and that 
$$\prod_{v\in \Bbb P \cup \{\infty\}} (a,b)_v=1.$$
This is a product of finitely many terms $(a,b)_v$, and  if $a$ or $b$ is a non-square, then the number of terms in the product that is not $1$ must be even (as the product is $1$). 
Thus $a$ or $b$ can be a non-square at only an even number of $v$.\end{proof}

Note that quadratic reciprocity implies that $\A_{\Q}$ containing $\Q$. We remark that the $\mathcal{L}^{fin,res}$-theory of $\A_K$ is stronger than the $\cL_{rings}$-theory of $\A_K$.

\

\subsection{\it Interpretable sets and imaginaries}

\

\

It is natural to ask whether given a family of $L$-stuctures $\mathcal{M}_i$, $i\in I$, in some language $L$ that have uniform elimination of imaginaries, the restricted product of $\mathcal{M}_i$ 
relative to some formula $\varphi(x)$ also has elimination of imaginaries in the language for restricted products induced by $L$ (see Section \ref{sec-rest}). 
This, in view of the uniform $p$-adic elimination of imaginaries due to Hrushovski-Martin-Rideau \cite{HMR} in a many-sorted language expanding the ring language by sorts for the spaces of lattices $GL_n(\Q_p)/GL_n(\Z_p)$ for all $n$, would give an elimination of imaginaries for $\A_{\Q}^{fin}$ (and hence for $\A_{\Q}$ in using elimination of imaginaries for archimedean local fields).

In \cite{CC2}, Connes-Consani have studied the space $\Q^{\times} \setminus \A_{\Q} / \hat{\Bbb Z}^*$ which is a quotient of the space of adele classes $\A_{\Q}/{\Q}^*$ (introduced by Connes) by the maximal compact subgroup of the idele class group $\hat{\Bbb Z}^*$. By works of Connes, these spaces are related to the zeroes of the Riemann zeta function $\zeta(s)$ (cf.~\cite{CC} and references there). We show that these spaces are imaginaries in $\A_{\Q}^f$.

\begin{prop} The structure $\Q^{\times} \setminus \A_{\Q} / \hat{\Bbb Z}^*$ is interpretable in $\A_{\Q}$.\end{prop}
\begin{proof} 
Let $T=\hat{\Bbb Z}^* \Bbb Q^*$. We claim that
\begin{equation}\label{(*)}
x\in T \Leftrightarrow \exists \theta \exists \alpha [([[\Phi_{val}(x) \wedge \Phi_{val}(x^{-1})]]^{na}=1 \wedge [[\alpha\neq 0]]^{na}=1 \wedge Fin([[\alpha\neq 1]]).
\end{equation}
Recall that $\Phi_{val}(x)$ is the formula from the language of rings that expresses that $v(x)\geq 0$ (cf. Section \ref{fin-supp}). 
The first conjunct on the right hand side expresses that $\alpha \in \hat{\Z}^*$. The second and third conjuncts express 
that $\alpha(p)$ is non-zero for all $p$, and that $\alpha(p)$ is $1$ for all but finitely many $p$, respectively.

Indeed, suppose $x\in T$. Then $x=\theta r$, where $\theta \in \hat{\Z}^*$ and $r\in \Q^*$. Then 
$$x(p)=\theta(p)r=\theta(p)u(p)p^{v_p(r)}=u'(p)p^{v_p(r)},$$
where $u'(p)$ is a unit in $\Z_p$ and $v_p$ denotes the $p$-adic valuation. 
We can write this as $\eta e$, where $\eta$ is in $\hat{\Z}^*$ and $e$ is an adele that is non-zero at all places $e(p)$, and $1$ at almost all places $e(p)$ (outside the finite set of prime divisors $p$ of $r$). 

Conversely, consider an element $x:=\theta \alpha$, where $\theta \in \hat{\Z}^*$ and $\alpha$ is an adele such that 
$\alpha(p)=1$ for all $p$ outside a finite set $E$ of primes, and $\alpha(p)$ is non-zero for all $p$. Define 
$r=\prod_{p\in E} p^{v_p(\alpha(p))}$. Then for all $p\in E$ we have that $r^{-1}x(p)=r^{-1}\theta(p)\alpha(p)$ is a 
unit in $\Z_p$ since $r^{-1}\alpha(p)$ is a unit in $\Z_p$ for all $p\in E$. For all $p\notin E$ we have that $r^{-1}x(p)=r^{-1}\theta(p)\alpha(p)=r^{-1}\theta(p)$ is a unit in $\Z_p$ since $r$ is a unit in $\Z_p$ for all $p\notin E$. So $r^{-1}x$ is in $\hat{\Z}^*$, and $x=r(r^{-1})x$ has the required form.
\end{proof}

\section{Uniformity in the number field $K$}

It is natural ask if the adelic quantifier elimination given in Theorem \ref{th-qe} is independent of the number field $K$, or the weaker question whether the theory of $\A_K$ for all $K$ is decidable. The first question states whether given an $L$-formula $\psi(\bar x)$ there exists a Boolean $\cL$-formula $\theta$ and $L$-formulae $\psi_1(\bar x),\dots,\psi(\bar x)$ that do not depend on $K$ and are quantifier free in a decidable language, such that for any $a_1,\dots,a_n\in \A_K^{fin}$
$$\Bbb A_K^{fin}\models \Psi(a_1,\dots,a_n) \Leftrightarrow$$
$$\B_{K}\models \theta([[\psi_1(a_1,\dots,a_n)]],\dots,[[\psi_m(a_1,\dots,a_n)]])).$$
Inspection of the proof of Theorem \ref{th-qe} and the works \cite{DM-supp} and \cite{FV} show that the construction of the formulae $\theta$ and $\psi_1,\dots,\psi_m$ do not depend on $K$ and are uniform for various families of $L$-structures and Boolean algebras. Thus if one has a quantifier elimination that was true uniformly for all finite extensions of $\Q_p$ for all $p$, then we would deduce that Theorem \ref{th-qe} holds uniformly for all $K$. 

The existence of such a uniform quantifier elimination in fact reduces to the case of finite extensions of $\Q_p$ for a 
single prime.
\begin{thm}\label{th-dec-ad} The theory of all sentences that hold in the adele rings $\A_K$ for all number fields $K$ 
is decidable if and only if for each fixed prime $p$, the theory of all finite extensions of $\Bbb Q_p$ is decidable. The adelic quantifier elimination in Theorem \ref{th-qe} does not depend on the number field $K$ if and only if the family of all finite extensions of $\Q_p$ have uniform quantifier elimination in $L$, for any given prime $p$ (where $L$ is as in Theorem \ref{th-qe}).\end{thm}
Indeed, let $\Phi$ be a sentence of the language of rings. Then it is possible to effectively find 
some prime $p_0$ depending only on 
$\Phi$ such that for any number field $K$ and any completion $K_v$ with residue field $k_v$, where $v$ is a 
non-archimedean valuation, if the characteristic of $k_v$ is at least $p_0$ then the truth of $\Phi$ in 
$K_v$ is decidable. This follows from a result of Ax-Kochen-Ershov but can also be deduced from the residue characteristic zero case of the Basarab-Kuhlmann quantifier elimination in \cite{basarab} or \cite{kuhlmann} for a valued field relative to residue rings (in this case all the higher residue rings are equal to the residue field). The value group is decidable and has quantifier elimination in a suitable language, so one only needs to decide residue field statements, and this can be done by Ax's work \cite{ax} (and moreover a suitable uniform quantifier elimination for the residue fields exists by Kiefe \cite{kiefe}).

The main obstacle to proving the uniform quantifier elimination stated in Theorem \ref{th-dec-ad} is the problem of 
quantifier elimination or decidability for the class of ramified extensions of $\Q_p$ of unbounded ramification index, or for infinitely ramified Henselian valued fields.

\bibliographystyle{acm}
\bibliography{bibadeles}

\end{document}